\documentclass[11pt]{article}

\usepackage{fullpage}
\usepackage[leqno]{amsmath}
\usepackage{amssymb}
\usepackage{amsthm}
\usepackage{graphicx}
\usepackage[usenames]{color}
\usepackage{wasysym}
\usepackage{enumerate}
\usepackage{verbatim}
\usepackage{rotate}
\usepackage{mathrsfs}
\usepackage{dsfont}
\usepackage{graphicx}
\usepackage{grffile}
\usepackage{epstopdf}
\usepackage[all,cmtip]{xy}

\newcommand{\Cliques}{\mathsf{Cliques}}
\newcommand{\MaxCliques}{\mathsf{MaxCliques}}

\newcommand{\components}{\mathsf{CC}}
\newcommand{\ProbDist}{\mathsf{Dist}}
\newcommand{\Hom}{\mathsf{Hom}}
\renewcommand{\hom}{\mathsf{hom}}
\newcommand{\hde}{\mathsf{HDE}}
\newcommand{\MRF}{\mathsf{MRF}}
\newcommand{\Supp}{\mathsf{Supp}}
\newcommand{\Vee}{\mathit{Vee}}
\newcommand{\Tri}{\vec{C}_3}

\renewcommand{\ge}{\geqslant}
\renewcommand{\le}{\leqslant}
\renewcommand{\geq}{\ge}
\renewcommand{\leq}{\le}

\newcommand{\ds}{\displaystyle}
\newcommand{\ts}{\textstyle}
\newcommand{\ul}{\underline}

\newcommand{\mr}{\mathrm}

\newcommand{\mc}{\mathcal}

\newcommand{\mds}{\mathds}

\newcommand{\uhr}{\upharpoonright}

\newcommand{\sq}[1]{\ensuremath{\langle#1\rangle}}

\newcommand{\wt}{\widetilde}

\newcommand{\N}{\mds{N}}
\newcommand{\BAR}[1]{\overline{#1}}
\newcommand{\longra}{\longrightarrow}

\renewcommand{\H}{\Ent}

\newcommand{\R}{\mds{R}}
\newcommand{\lra}{\longrightarrow}
\newcommand{\Ent}{\mathbb{H}}

    \newtheorem{thm}{Theorem}[section]
    \newtheorem{next-thm}[thm]{Theorem*}
    
    \newtheorem{next-prop}[thm]{Proposition*}
    \newtheorem{la}[thm]{Lemma}
    \newtheorem{cor}[thm]{Corollary}
    
    \newtheorem{claim}[thm]{Claim}

\theoremstyle{definition}
    \newtheorem{df}[thm]{Definition}

    \newtheorem{ex}[thm]{Example}

\begin{document}

\title{The Homomorphism Domination Exponent}
\author{Swastik Kopparty\thanks{Computer Science and Artificial Intelligence
Laboratory, MIT {\tt swastik@mit.edu}.}
\and Benjamin Rossman\thanks{Computer Science and Artificial Intelligence
Laboratory, MIT {\tt brossman@mit.edu}. Supported by the National Defense Science and Engineering Graduate Fellowship.
}}

\maketitle{}

\begin{abstract}
We initiate a study of the {\em homomorphism domination exponent}
of a pair of graphs $F$ and $G$, defined as the maximum
real number $c$ such that $|\Hom(F,T)| \ge |\Hom(G,T)|^c$ for every
graph $T$. The problem of determining whether $\hde(F,G) \ge 1$ is
known as the {\em homomorphism domination problem} and its
decidability is an important open question arising in the theory of
relational databases. We investigate the combinatorial and
computational properties of the homomorphism domination exponent,
proving upper and lower bounds
and isolating classes of graphs $F$
and $G$ for which $\hde(F,G)$ is computable. 
In particular, we present a linear program computing $\hde(F,G)$ in the
special case where $F$ is chordal and $G$ is series-parallel.
\end{abstract}

\newpage

\tableofcontents{}

\section{Introduction}

A well known corollary of the Kruskal-Katona theorem states that
a graph with $e$ edges can have at most $e^{3/2}$ triangles. More
generally one may ask: given two graphs $F$ and $G$, if we know that a third
graph $T$ has $a$ copies of $F$ as a subgraph, what can we say about the number
of copies of $G$ in $T$? This paper is an attempt to pursue a systematic study of a
general question of this type.

For (directed) graphs $F$ and $G$, a {\em homomorphism from $F$ to
$G$} is a function $\varphi$ from the vertices of $F$ to the
vertices of $G$ such that for any edge $(u,v)$ of $F$, the pair
$(\varphi(u), \varphi(v))$ is an edge of $G$. The set of all
homomorphisms from $F$ to $G$ is denoted $\Hom(F,G)$, its
cardinality is denoted $\hom(F,G)$, and we write $F \to G$ if
$\hom(F,G) \ge 1$.

Given a graph $T$, one can consider the profile of its ``subgraph counts" given by
the numbers $\hom(F,T)$, as $F$ varies over all finite graphs.
The set of all possible profiles encodes much information about the local stucture
of graphs. This motivates the following central meta-question in graph theory: find all relations that
the numbers $\hom(F_1, T), \ldots, \hom(F_t, T)$ must satisfy in every graph $T$.
Unfortunately, a satisfactory understanding of these relations has thus far been elusive.
This failure is explained by the following simple but striking result (due to
Ioannidis and Ramakrishnan~\cite{IR}, discovered in the context of theoretical databases): 
given graphs $F_1, \ldots, F_t$ and integers
$a_1, \ldots, a_t$, it is undecidable whether for all graphs $T$, the
following inequality holds:
$$\sum_{i = 1}^t a_i \hom(F_i, T) \geq 0.$$
The undecidability (via a reduction to Hilbert's 10th Problem) already holds if we restrict $t = 9$.
Thus, one cannot hope to fully understand the relative magnitudes
of subgraph counts of even just $9$ graphs at a time!
Given this unfortunate fact, we set our sights a little lower, and
attempt to study the relative homomorphism numbers from two graphs.

For graphs $F$ and $G$ such that $F \to G$, the {\em homomorphism
domination exponent} of $F$ and $G$, denoted $\hde(F,G)$, is defined
as the maximal real number $c$ such that $\hom(F,T) \ge \hom(G,T)^c$
for all ``target" graphs $T$. The $\hde$ is a parameter encoding deep
aspects of the local structure of graphs, and we believe that it is worthy
of further study. As a concrete goal, here we consider the question of
computing $\hde(F,G)$ given graphs $F$ and $G$.

Another motivation for the $\hde$ comes from the theory
of databases. The {\em containment problem for conjunctive queries} ({\em under
multiset semantics}), a problem of much importance in database theory,
is equivalent to the {\em homomorphism domination problem} in graph theory
which asks, given graphs $F$ and $G$, whether $\hom(F,T) \ge \hom(G,T)$ for all graphs
$T$. The homomorphism domination exponent is a quantitative version of the
homomorphism domination problem (or the conjunctive query containment problem);
note that the homomorphism domination problem is simply the question whether \mbox{$\hde(F,G) \ge 1$.}

Many classical inequalities involving graphs are naturally viewed in terms of the homomorphism domination exponent. For example, the Kruskal-Katona Theorem determines the maximum number of triangles in a graph with a given number of edges. This relationship is captured by the equality
$\hde(\,\raisebox{-2pt}{\includegraphics{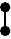}}\,,\raisebox{-2pt}{\includegraphics{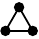}}) = 2/3$. Similarly, a result of K\"ov\'ari, S\'os and Tur\'an~\cite{KST}, which establishes a relationship between the numbers of vertices, edges and $4$-cycles in a graph $G$, states that $\hom(C_4,G) \geq \big(\hom(\,\raisebox{-2pt}{\includegraphics{edge.pdf}}\,,G)/\hom({\scriptstyle\bullet}\,,G)\big){}^4$.
This is summarized by the inequality
$\hde( C_4 + {\scriptstyle\bullet\,\bullet\,\bullet\,\bullet\,} 
, \,\raisebox{-2pt}{\includegraphics{edge.pdf}}\,) \geq 4$.
In Section~\ref{sub:known} we give an overview of known results from extremal
combinatorics that imply general bounds on the homomorphisms domination
exponent.

Our principal objective in this paper is to give algorithms for
computing and bounding the homomorphism domination exponent. We
introduce new combinatorial techniques for proving inequalities
between homomorphism numbers and establishing their tightness.

\subsection{Overview of Results}\label{sec:overview}

We prove a lower bound on $\hde(F,G)$ when $F$ is chordal and $G$ is any graph such that $F \to G$.  This lower bound has the form of a linear program over the convex set of $G$-polymatroidal functions (defined in Section~\ref{sec:G-polymat}).  In the special case where $F$ is chordal and $G$ is series-parallel, this linear program computes $\hde(F,G)$ exactly.  A relaxation of this linear program turns out to be an upper bound on $\hde(F,G)$ for all graphs $F$ and $G$.  These results are stated formally in Section~\ref{sec:results}.

Our bounds yield several new inequalities for graph homomorphism numbers. For instance:
\begin{gather*}
  \hde\left(\raisebox{-6pt}{\includegraphics{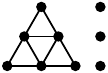}}\,,\,
  \raisebox{-6pt}{\includegraphics{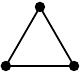}}\right)
  = \frac{5}{2},\\
  \hde\big(\text{any directed tree of size }n,\,
  \text{the directed $n$-cycle }\vec C_n\big) = 1.
\end{gather*}
Let $P_n$ denote the undirected path of size $n$ (with $n$ vertices and $n-1$ edges). Our main theorem implies:
\begin{align*}
	&&&&\hde(P_m,P_n) &= 1 &&\text{when }m \ge n,&&&&\\
	&&&&\hde(P_m,P_n) &= m/n &&\text{when $m \le n$ and $m$ is odd.}&&&&
\end{align*}
However, when $m \le n$ and $m$ is even, the value of $\hde(P_m,P_n)$ is slightly less than $m/n$ (by an amount that depends on $n \mod m$):
\begin{align*}
    \hde(P_2,P_n) &= 
    		1/\lceil n/2 \rceil,\\
	\ds\hde(P_4,P_{4n+i}) &= \begin{cases}
        1/n 
        		&\text{if }i=0,\\
        2/(2n+1) 
        		&\text{if }i=1,\\
        (4n+1)/(4n^2+3n+1) 
        		&\text{if }i=2,\\
        1/(n+1)
        		&\text{if }i=3.
    \end{cases}
\end{align*}
These expressions were discovered by solving the linear program in our main theorem for small values of $n$ (which then suggested proofs for arbitrary $n$).  The equation $\hde(P_4,P_{4n+2}) = (4n+1)/(4n^2+3n+1)$ (stated as Theorem~\ref{thm:P_4}) in particular stands out as an example of an intriguing phenomenon associated with the HDE. Its proof (included in \S\ref{sec:P_4}) seems like it might be hard to come up with by hand. We remark that finding a closed expression for $\hde(P_m,P_n)$ for all $m$ and $n$ is an open problem.

By contrast, $\hde(C_m,C_n)$ for cycles $C_m$ and $C_n$ contains no surprises. An anonymous referee pointed out that H\"older's inequality implies that $\hde(C_m,C_n) = \min(m/n,1)$ in all cases when $C_m \to C_n$ (i.e.,\ $m$ is even or $n$ is odd and $m \ge n$). 

Finally, we mention that our results (Theorem~\ref{thm:lower}) can be used to give another proof---using entropy methods---of Sidorenko's conjecture \cite{Sid} for the special case of forests.

\subsection{The Method via an Example}\label{sub:method}

We prove our bounds using an approach based on entropy and linear
programming. We now briefly illustrate our methods in action on a
simple example. The argument is inspired by the entropy proof of
Shearer's lemma, often attributed to Jaikumar Radhakrishnan, and its generalizations
due to Friedgut and Kahn~\cite{FK,Friedgut}.

Consider the graphs $\Vee$ and $\Tri$ pictured below.
\[
    \includegraphics[scale=.85]{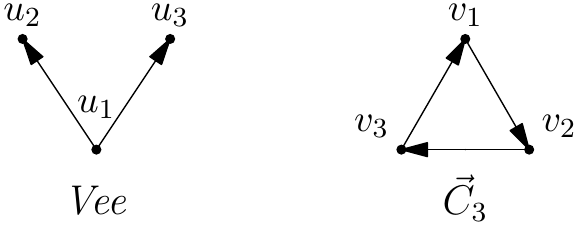}
\]
We will prove that $\hde(\Vee,\Tri) = 1$. (This problem was posed by
Erik Vee~\cite{Vee}; a different solution and generalization were
given by Rossman and Vee~\cite{RV}.) As $\hom(\Vee,\Tri) =3$ and
$\hom(\Tri,\Tri) = 3$, we have $\hde(\Vee,\Tri) \leq 1$. It remains
to show that for all graphs $T$, $\hom(\Vee,T) \geq \hom(\Tri,T)$.
To that end, fix an arbitrary graph $T$ such that $\Tri \to T$. Pick
$\chi$ uniformly at random from $\Hom(\Tri,T)$. For $i = 1,2,3$, let
$a_i = \chi(v_i)$. Observe that the joint distribution $(a_1, a_2,
a_3)$ is uniform on a subset of $V_T \times V_T \times V_T$ of size
$\hom(\Tri,T)$. Thus $\Ent(a_1, a_2, a_3) = \log \hom(\Tri,T)$. We
now prove that $\Ent(a_1, a_2, a_3) \leq \log \hom (\Vee,T)$.

By the chain rule of entropy,
$$\Ent(a_1, a_2, a_3) = \Ent(a_1) + \Ent(a_2 | a_1) + \Ent(a_3 | a_1, a_2).$$
As conditioning on fewer variables can only increase entropy, we get
$$\Ent(a_1, a_2, a_3) \leq \Ent(a_1) + \Ent(a_2 | a_1) + \Ent(a_3 | a_2).$$
Now, by cyclic symmetry of $a_1, a_2, a_3$, we have $\Ent(a_3 | a_2)
= \Ent(a_2 | a_1)$. Thus,
\begin{equation}
\label{tri-vee} \Ent(a_1, a_2, a_3) \leq \Ent(a_1) + 2 \Ent(a_2 |
a_1).
\end{equation}
We will now interpret this expression. Consider the distribution
$(x,y,y')$ on $V_T \times V_T \times V_T$ defined as follows. First, $x
\in V_T$ is picked according to the distribution of $a_1$. Next, two
independent copies $y, y' \in V_T$ of $a_2$ {conditioned on $a_1 = x$}
are picked. The entropy of $(x,y,y')$ is easily computed:
$$\Ent(x,y,y') = \Ent(x) + \Ent(y|x) + \Ent(y'|x) = \Ent(a_1) + \Ent(a_2|a_1) + \Ent(a_2 | a_1).$$
Thus, we have $\Ent(a_1, a_2, a_3) \leq \Ent(x,y,y')$ by
\eqref{tri-vee}.

Distribution $(x, y, y')$ was constructed so that there is always an
edge from $x$ to $y$ as also from $x$ to $y'$. Thus, every point of
$V_T \times V_T \times V_T$ in the support of the distribution of
$(x,y,y')$ specifies a unique homomorphism in $\Hom(\Vee,T)$, namely
the map $u_1 \mapsto x$, $u_2 \mapsto y$ and $u_3 \mapsto y'$.
This implies that $\log\hom(\Tri,T) =\Ent(a_1,a_2,a_3) \leq \log
\hom(\Vee,T)$, completing the proof.

The proof of our lower bound on $\hde(F,G)$ for chordal graphs $F$
and arbitrary graphs $G$ follows the same strategy as the argument
above. When we want to prove that for all $T$, $\hom(F,T) \geq
\hom(G,T)^c$, we start with a uniform distribution on $\Hom(G,T)$.
We analyze its entropy and compare it with the entropy of several
auxiliary distributions that we construct on $\Hom(F,T)$. The
construction of the auxiliary distributions, as well as the analysis
and comparisons of entropies are guided by a linear program.

\subsection{Related Work}\label{sub:known}

Several computational problems closely related to the computability of
the homomorphism domination exponent are known to be undecidable.
Validity of linear inequalities
involving homomorphism numbers was shown to be undecidable by
\cite{IR} via a reduction from Hilbert's 10th problem on solvability of
integer diophantine equations. The homomorphism domination problem
with ``inequality constraints" is also known to be undecidable~\cite{JKV}.

Inequalities between homomorphism numbers have been extensively
studied in extremal combinatorics. For a survey, see \cite{BCLSV}.
Very few general results are known about the homomorphism domination exponent (defined here for the first time, but implicitly studied before). Alon
\cite{Alon} showed that if $e$ is an undirected edge and $G$ is any simple
graph, then $\hde(e, G) = \frac{1}{\rho(G)}$, where $\rho(G)$ is the
{\em fractional edge covering number} of $G$. This result was
reproved and generalized to hypergraphs by Friedgut and
Kahn~\cite{FK}. Their argument used Shearer's lemma, which is
closely related to the entropy techniques that we use. A wonderful
exposition on using entropy and Shearer's lemma to prove classical
inequalities can be found in \cite{Friedgut}. Galvin and Tetali
\cite{GT}, generalizing an argument of Kahn \cite{Kahn}, also using
entropy techniques, showed that for any $n$-regular, $N$-vertex
bipartite graph $G$, $\hde(K_{n,n}, G) = \frac{2n}{N}$.
Finally, a very general approach to inequalities between homomorphism
numbers in dense graphs was developed in~\cite{BCLSV, Razborov}.
However, it is not known whether this approach can yield algorithms
for deciding validity of special families of inequalities between homomorphism numbers.

The entropy arguments that we use differ from the above applications
in that we utilize finer information about conditional entropy. The
key technical device that enables us to use this information is the
construction of auxiliary distributions using conditionally
independent copies of the same random variable. This is exemplified
in the example of the previous subsection by our definition of the
distribution $(x,y,y')$.

\paragraph{\bf Paper Organization.}

Section~\ref{sec:prelims} introduces the necessary definitions and tools related to graphs and homomorphisms. Our results are formally stated in Section~\ref{sec:results}. Definitions and auxiliary lemmas on Markov random fields are given in Section~\ref{sec:MRF}. Proofs of our main theorems are presented in Sections~\ref{sec:lb}, \ref{sec:ub}, \ref{sec:tight} and \ref{sec:P_4}. We state our conclusions in Section~\ref{sec:conclusion}.

\section{Preliminaries}\label{sec:prelims}

We first fix some basic notation.
    For a natural number $n$, let $[n]$ denote the set
    $\{1,\dots,n\}$.
    The powerset of a set $X$ is denoted by $\wp(X)$.
    If $\mc S$ is a family of sets, let $\bigcap \mc S$ denote the
    intersection $\bigcap_{S \in \mc S} S$. We adopt the convention that $\bigcap\emptyset = \emptyset$.

\subsection{Graphs and Homomorphisms}

{\em Graphs} will be finite and directed. Formally, a graph is a
pair $G = (V_G,E_G)$ where $V_G$ is a nonempty finite set and $E_G$
is a subset of $V_G \times V_G$. For a subset $A \subseteq V_G$, we
denote by $G|_A$ the induced subgraph of $G$ with vertex set $A$. We
denote by $k{\cdot}G$ the disjoint union of $k$ copies of $G$. The
{\em (categorical) product} $F \times G$ of graphs $F$ and $G$ has
vertex set $V_{F\times G} = V_F \times V_G$ and edge set $E_{F\times
G} = \{((a,v),(b,w)) : (a,b) \in E_F$ and $(v,w) \in E_G\}$.

A graph $G$ is {\em simple} if the relation $E_G$ is antireflexive
and symmetric, i.e., if $(v,w) \in E_G$ then $v \ne w$ and $(w,v)
\in E_G$. Every graph $G$ is associated with a simple graph $\BAR G$
defined by $V_{\BAR G} = V_G$ and $E_{\BAR G} = \{(v,w) : v \ne w$
and $(v,w) \in E_G$ or $(w,v) \in E_G\}$. Whenever we speak of
cliques, connectivity, etc., of $G$, we mean cliques, connectivity,
etc., of the associated simple graph $\BAR G$. In particular, a {\em
clique} in a graph $G$ is a set of vertices $A \subseteq V_G$ such
that $(v,w) \in E_G$ or $(w,v) \in E_G$ for all distinct $v,w \in
A$. We denote by $\Cliques(G)$ the set of cliques in $G$ and by
$\MaxCliques(G)$ the set of maximal cliques in $G$. The number of
connected components of $G$ is denoted by $\components(G)$.

A {\em homomorphism} from a graph $F$ to a graph $G$ is a function
$\varphi : V_F \longra V_G$ such that $(\varphi(a),\varphi(b)) \in
E_G$ for all $(a,b) \in E_F$. Let $\Hom(F,G)$ denote the set of
homomorphisms from $F$ to $G$ and let $\hom(F,G) = |\Hom(F,G)|$.
Notation $F \to G$ expresses $\hom(F,G) \ge 1$. Under disjoint
unions ($+$) and categorical graph product ($\times$), $\hom(\ul{\
\,},\ul{\ \,})$ obeys identities
\begin{align*}
    \hom(F_1 + F_2, G) &= \hom(F_1, G) \cdot \hom(F_2,G),\\
    \hom(F, G_1 \times G_2) &= \hom(F, G_1) \cdot \hom(F, G_2).
\end{align*}

A graph $F$ is {\em chordal} if the simple graph $\BAR F$ contains
no induced cycle of size $\ge 4$. Chordal graphs are alternatively
characterized by the existence of an elimination ordering. A vertex
$v$ is {\em eliminable} in a graph $F$ if the neighborhood of $v$ is a clique in $F$.
An enumeration $v_1,\dots,v_n$ of $V_F$ is an {\em elimination ordering} for $F$ if
$v_j$ is eliminable in $F|_{\{v_1,\dots,v_j\}}$ for all $j \in [n]$.
By a well-known characterization, a graph $F$ is chordal if and only
if it has an elimination ordering.

A {\em $2$-tree} is a chordal graph with clique number at most $3$
(i.e., containing no $K_4$). A graph $G$ is {\em series-parallel} if
$G$ is a subgraph of some $2$-tree.

\subsection{The Homomorphism Domination Exponent}

We now formally define the homomorphism domination exponent.

\begin{df}[Homomorphism Domination Exponent]\label{df:hde-def}
For graphs $F$ and $G$ such that $F \to G$,\footnote{We do not
define $\hde(F,G)$ whenever $F \not\to G$. However, it might be a
reasonable convention to let $\hde(F,G) = -\infty$.} the {\em
homomorphism domination exponent} $\hde(F,G)$ is defined by
\begin{equation}\notag
    \hde(F,G) = \sup\big\{c \in \R : \hom(F,T) \ge \hom(G,T)^c
    \text{ for all graphs }T\big\}.
\end{equation}
We write $F \succcurlyeq G$ and say $F$ {\em homomorphism-dominates}
$G$ if $\hde(F,G) \ge 1$.
\end{df}

The following dual expression for $\hde(F,G)$ is often useful:
\begin{equation}
\label{hdeotherdef}
    \hde(F,G) = \inf_{T\,:\,\hom(G,T) \ge 2} \frac{\log\hom(F,T)}{\log\hom(G,T)}.
\end{equation}
We remark that this $\inf$ is not always a $\min$.

The following lemma (proof omitted) lists some basic properties of the homomorphism
domination exponent.

\begin{la}[Basic Properties of $\hde$]\label{la:homdom-facts}
\
\begin{enumerate}[\normalfont\hspace{\parindent}(a)]\setlength{\itemsep}{2pt}
  \item
    If $c = \hde(F,G)$, then $\hom(F,T) \ge \hom(G,T)^c$ for all
    graphs $T$. (That is, we can replace $\sup$ by $\max$ in
    Definition~\ref{df:hde-def}.)
  \item
    The homomorphism-domination relation $\succcurlyeq$ is a partial
    order on graphs.
  \item
    $\hde(F,H) \ge \hde(F,G)\cdot\hde(G,H)$.
  \item
    $\hde(m{\cdot} F,n{\cdot} G) = \frac m n \cdot \hde(F,G)$ for all 
    positive integers $m,n$.
  \item
    If there exists a surjective homomorphism from $F$ onto $G$,
    then $F \succcurlyeq G$.
  \item
    $\hde(F,G) > 0$ if and only if $\bigcup_{\varphi \in
    \Hom(F,G)} \mr{Range}(\varphi) = V_G$.
\end{enumerate}
\end{la}

By \eqref{hdeotherdef}, every graph $T$ with $\hom(G, T) \ge 2$
provides an upper bound on $\hde(F,G)$. By taking specific graphs
$T_1$, $T_2$ and $(T_{3,n})_{n \ge 1}$ in the figure below, we get
the following general upper bounds on $\hde(F,G)$.
\begin{center}
    \includegraphics{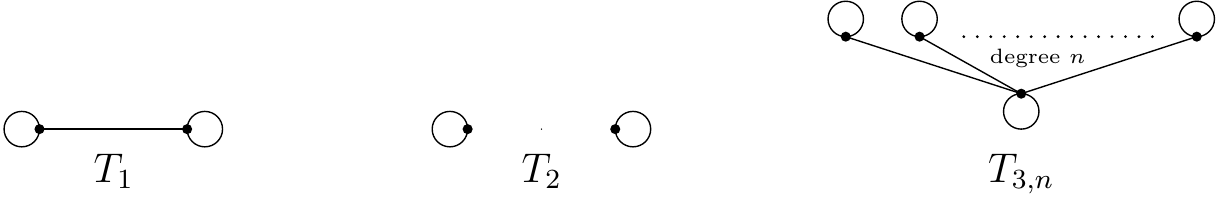}
\end{center}
Taking $T = T_1$, we get the upper bound $\hde(F,G) \leq |V_F| /
|V_G|$. Taking $T = T_2$, we have that $\hde(F,G) \leq
\components(F)/ \components(G)$. A slightly more complicated upper
bound follows by taking $T = T_{3,n}$ and letting $n \to \infty$;
the result is that $\hde(F,G)$ is at most the ratio
$\alpha(F)/\alpha(G)$ of the independence numbers of $F$ and $G$,
since $\hom(H,T_{3,n})$ grows like $\Theta(n^{\alpha(H)})$ for every
graph $H$.

\subsection{$G$-Polymatroidal Functions}\label{sec:G-polymat}

\begin{df}\label{df:G-polymat}
For a graph $G$, let $\mc P(G)$ and $\mc Q(G)$ be the following sets of functions from $\wp(V_G)$ to $[0,1]$.
\begin{itemize}\setlength{\itemsep}{4pt}
  \item
	A function $p : \wp(V_G) \lra \R$ is {\em $G$-polymatroidal} if it satisfies the following four conditions:
	\begin{enumerate}[] 
	  \item
	  	\makebox[1.3in]{($0$ at $\emptyset$)\hfill} $p(\emptyset) = 0$,
	  \item
		\makebox[1.3in]{(monotone)\hfill} $p(A) \le p(B)$ for all $A \subseteq B \subseteq V_G$,
	  \item
	  	\makebox[1.3in]{(submodular)\hfill} $p(A \cap B) + p(A \cup B) \le p(A) + p(B)$ for all $A,B \subseteq V_G$,
	  \item
	  	\makebox[1.3in]{($G$-independent)\hfill} 
		\parbox[t]{3.35in}{$p(A \cap B) + p(A \cup B) = p(A) + p(B)$ for all $A,B \subseteq V_G$  such that $A \cap B$ separates $A \setminus B$ and $B \setminus A$ in $G$ (i.e., there is no edge in $G$ between $A \setminus B$ and $B \setminus A$).}
	\end{enumerate}
	A $G$-polymatroidal function $p$ is {\em normalized} if in addition it satisfies:
	\begin{enumerate}[] 
	  \item
	  	\makebox[1.3in]{(normalized)\hfill} $p(V_G) = 1$.
	\end{enumerate}
  \item
	$\mc P(G)$ denotes the set of normalized $G$-polymatroidal functions.
  \item
  	$\mc Q(G)$ denotes the set of functions $q : \wp(V_G) \lra \R$ which satisfy:
\[
	q(\emptyset) = 0, \qquad q(A) \ge 0\text{ for all }A \subseteq V_G, \qquad \sum_{A \subseteq V_G} q(A) \cdot \components(G|_A) = 1.
\]
\end{itemize}
\end{df}

\begin{ex}\label{ex:K_3}
Let $a,b,c$ be the vertices of $K_3$. Then $\mc P(K_3)$ is the set of convex combinations of eight functions from $\wp(\{a,b,c\})$ to $[0,1]$, which we label as $f_a,f_b,f_{ab},f_{ac},f_{bc},f_{abc}$ (corresponding to the seven
nonempty subsets of $\{a,b,c\}$) and $f_{\mr{RS}}$ (``\textrm{RS}'' stands for Ruzsa-Szemer\'edi, for reasons that will be explained later on), given by the following table:\upshape
\[
\begin{tabular}{l||c|c|c|c|c|c|c|c|}
   \parbox{38pt}{\mbox{}}
 & \parbox{38pt}{\mbox{}\hfill$\emptyset$\hfill\mbox{}}
 & \parbox{38pt}{\mbox{}\hfill$\{a\}$\hfill\mbox{}}
 & \parbox{38pt}{\mbox{}\hfill$\{b\}$\hfill\mbox{}}
 & \parbox{38pt}{\mbox{}\hfill$\{c\}$\hfill\mbox{}}
 & \parbox{38pt}{\mbox{}\hfill$\{a,b\}$\hfill\mbox{}}
 & \parbox{38pt}{\mbox{}\hfill$\{a,c\}$\hfill\mbox{}}
 & \parbox{38pt}{\mbox{}\hfill$\{b,c\}$\hfill\mbox{}}
 & \parbox{38pt}{\mbox{}\hfill$\{a,b,c\}$\hfill\mbox{}}\\
  \hline\hline\quad$f_{a}$& 0 & 1 & 0 & 0 & 1 & 1 & 0 & 1\\
 \hline\quad$f_{b}$& 0 & 0 & 1 & 0 & 1 & 0 & 1 & 1\\
 \hline\quad$f_{c}$& 0 & 0 & 0 & 1 & 0 & 1 & 1 & 1\\
 \hline\quad$f_{ab}$& 0 & 1 & 1 & 0 &1 &1 & 1&1\\
 \hline\quad$f_{ac}$& 0 & 1 & 0 & 1 &1 &1 & 1&1\\
 \hline\quad$f_{bc}$& 0 & 0 & 1 & 1 &1 &1 & 1&1\\
 \hline\quad$f_{abc}$& 0 &1&1&1&1&1&1&1\\
 \hline\quad$f_{\mr{RS}}$ & 0 & $1/2
 $ & $1/2
 $ & $1/2
 $ &
 1&1&1&1\\
 \hline
\end{tabular}
\]
\end{ex}

We will use the following identity for $G$-polymatroidal functions when $G$ is chordal. 
\begin{la}[Identity for Chordal-Polymatroidal Functions]\label{la:entropy-of-chordal-MRF} 
If $G$ is chordal, then for every $G$-polymatroidal function $p : \wp(V_G) \lra \R$ and every elimination ordering $v_1,\dots,v_n$ for $G$,
\begin{align*}
    p(V_G) &= \sum_{S \subseteq \MaxCliques(G)} -(-1)^{|S|} p({\ts\bigcap}S)\\
    &= \sum_{i=1}^n p\big(
    \{\textup{neighbors of $v_i$ among }v_1,\dots,v_{i-1}\}
    \cup \{v_i\}\big) - p\big(
    \{\textup{neighbors of $v_i$ among }v_1,\dots,v_{i-1}\}\big).
\end{align*}
\end{la}
Lemma~\ref{la:entropy-of-chordal-MRF} is established by a straightforward inductive argument (proof omitted).

\section{Results}\label{sec:results}

Our first theorem gives a lower bound on $\hde(F,G)$ when $F$ is
chordal.

\begin{thm}\label{thm:lower}
If $F$ is chordal and $G$ is any graph, then
\begin{align*}
  \hde(F,G) &\ge \min_{p \in \mc P(G)} \max_{\varphi \in
  \Hom(F,G)} \sum_{S \subseteq
  \MaxCliques(F)} -(-1)^{|S|} \cdot p(\varphi({\ts\bigcap} S)).
\end{align*}
\end{thm}

Theorem~\ref{thm:lower} is proved by a generalization of the entropy
technique illustrated by the example in \S\ref{sub:method}. 

Our second theorem gives an upper bound on $\hde(F,G)$ for general
graphs $F$ and $G$.

\begin{thm}\label{thm:upper}
For all graphs $F$ and $G$,
\begin{align*}
\ds  \hde(F,G) \le
  \min_{q \in \mc Q(G)}
  \max_{\varphi \in \Hom(F,G)}
  \sum_{A\subseteq V_G} q(A)\cdot \components(F|_{\varphi^{-1}(A)})
\end{align*}
\end{thm}

The next theorem establishes that Theorem~\ref{thm:lower} is tight
in the special case where $G$ is series-parallel. 

\begin{thm}\label{thm:tight}
If $F$ is chordal and $G$ is series-parallel, then 
\begin{align*}
  \hde(F,G) &= \min_{p \in \mc P(G)} \max_{\varphi \in
  \Hom(F,G)} \sum_{S \subseteq
  \MaxCliques(F)} -(-1)^{|S|} \cdot p(\varphi({\ts\bigcap} S)).
\end{align*}
\end{thm}

The final theorem (mentioned in the introduction) is an example of an interesting HDE computation discovered with the help of the linear program of Theorem~\ref{thm:tight}.

\begin{thm}\label{thm:P_4}
$\ds\hde(P_4,P_{4n+2}) = \frac{4n+1}{4n^2+3n+1}$
\end{thm}

Theorems~\ref{thm:lower}, \ref{thm:upper}, \ref{thm:tight} and \ref{thm:P_4} are respectively
proved in Sections~\ref{sec:lb}, \ref{sec:ub}, \ref{sec:tight} and \ref{sec:P_4}.

\subsection*{Discussion 1. Tightness of our lower and upper bounds}

The HDE upper bound of Theorem~\ref{thm:upper} is not tight for all pairs of graphs. For instance, $F = C_4 + 2{\cdot}K_1$ (an undirected 4-cycle plus two isolated vertices) and $G = K_2$, it holds that $\hde(F,G) = 8/3$, while Theorem~\ref{thm:upper} only implies $\hde(F,G) \le 3$. However, we can show that Theorem~\ref{thm:upper} is tight when (the underlying simple graphs of) $F$ and $G$ are forests.

We do not have any example of a chordal graph $F$ and a graph $G$ for which the HDE lower bound of Theorem~\ref{thm:lower} is not tight. However, there are reasons to believe that the tightness of this lower bound is not the question. Recall that the linear program in Theorem~\ref{thm:lower} has domain $\mc P(G)$, the set of normalized $G$-polymatroidal functions. In fact (as will obvious from the proof of Theorem~\ref{thm:lower}), we can replace $\mc P(G)$ with the subset $\{h_X : X \in \MRF(G)\}$ of normalized entropic functions of Markov random fields over $G$ (defined in the next section). Let $\mc E(G)$ denote the closure of $\{h_X : X \in \MRF(G)\}$ in $\R^{V_G}$. The set $\mc E(G)$, whose members are called {\em $G$-entropic functions},
is a convex subset of $\mc P(G)$ and a well-studied object in information theory.  
When $|V_G| \le 3$, we have $\mc E(G) = \mc P(G)$. However, these sets do not coincide in general. For instance, $\mc E(K_4)$ is a proper subset of $\mc P(K_4)$ (due to the existence of ``non-Shannon information inequalities'' on $4$ random variables); in fact, $\mc E(\mc K_4)$ fails even to be a polytope. While it seems unnatural to conjecture that the HDE lower bound of Theorem~\ref{thm:lower} is tight as stated, the same conjecture for the corresponding linear program over $\mc E(G)$ would appear more reasonable.

\subsection*{Discussion 2. Theorem~\ref{thm:upper}  
is a linear program relaxation 
of Theorem~\ref{thm:lower}} 

It is worth pointing out that the linear program 
in the HDE upper bound of Theorem~\ref{thm:upper} is (after a linear change of variables) a direct relaxation of the linear program 
in the HDE lower bound of Theorem~\ref{thm:lower}. To see this,
consider the invertible linear transformation $L : \R^{\wp(V_G)} \lra \R^{\wp(V_G)}$ which takes a function $f : \wp(V_G) \lra \R$ to a function $Lf : \wp(V_G) \lra \R$ defined by
\[
		(Lf)(A) = \sum_{B 
		\,:\, A \cup B = V_G} -(-1)^{|A \cap B|} f(B).
\]
We need a combinatorial lemma on chordal graphs. 

\begin{la}\label{la:chordal-graph-identity}
Suppose $F$ is chordal.
\begin{enumerate}[\normalfont\hspace{\parindent}(a)]\setlength{\itemsep}{0pt}
\item
For all $A \subseteq V_F$,
\[
	\sum_{S \subseteq \MaxCliques(F)} (-1)^{|S|} = \sum_{B
	\,:\, A \cup B = V_F} (-1)^{|A \cap B|} \components(F|_B).
\]
\item
For every function $f : \wp(V_F) \lra \R$,
\[
	 \sum_{S \subseteq \MaxCliques(F)} -(-1)^{|S|} f({\ts\bigcap} S) = \sum_{A \subseteq V_F} (Lf)(A) \cdot \components(F|_A).
\]
\item
For every homomorphism $\varphi : F \lra G$ and function $g : \wp(V_G) \lra \R$,
\[
	 \sum_{S \subseteq \MaxCliques(F)} -(-1)^{|S|} g(\varphi({\ts\bigcap} S)) = \sum_{A \subseteq V_G} (Lg)(A) \cdot 	\components(F|_{\varphi^{-1}(A)}).
\]
\end{enumerate}
\end{la}

Lemma~\ref{la:chordal-graph-identity} can be proved by an inductive argument, or alternatively, 
using elementary algebraic topology (Euler characteristics of flag complexes associated with chordal graphs). Statement (a) is the essential identity; statement (b) follows directed from (a); statement (c), which is the result we need, is a slight extension of (b).

As an immediate corollary of Lemma~\ref{la:chordal-graph-identity}(c), we get:
\begin{cor}[Alternative Statement of Theorem~\ref{thm:lower}]
If $F$ is chordal and $G$ is any graph, then
\begin{align*}
  \hde(F,G) &\ge \min_{q \in L (\mc P(G))} \max_{\varphi \in \Hom(F,G)}
  \sum_{A \subseteq V_G} q(A) \cdot \components(F|_{\varphi^{-1}(A)}).
\end{align*}
\end{cor}
To see that the linear program of Theorem~\ref{thm:upper} is a direct relaxation of the linear program of Theorem~\ref{thm:lower}, it suffices to show that $\mc Q(G) \subseteq L(\mc P(G))$ for all graphs $G$, which can be checked by applying $L^{-1}$ to an arbitrary function in $\mc Q$ and seeing that the resulting function is normalized $G$-polymatroidal. Indeed, for any $q \in \mc Q(G)$, the function $L^{-1}q$ is given by $(L^{-1}q)(A) = 
\sum_{B \subseteq V_G} q(B) \cdot \components(G|_{\varphi^{-1}(A \cap B)})$,
which one can show is normalized $G$-polymatroidal.

\section{Chordal Pullbacks of Markov Random Fields}\label{sec:MRF}

A {\em (probability) distribution} over a nonempty finite set
$\Omega$ is a function $X : \Omega \lra [0,1]$ such that
$\sum_{\omega \in \Omega} X(\omega) = 1$. We denote by
$\ProbDist(X)$ the set of all distributions over $\Omega$. The {\em
support} of $X$ is the set $\Supp(X) = \{\omega \in \Omega :
X(\omega)> 0\}$. The {\em entropy} of $X$ is defined by $\H(X) =
\sum_{\omega \in \Omega} -X(\omega)\log X(\omega)$. Since the
uniform distribution maximizes entropy among all distributions with
a given support, it holds that $\H(X) \le \log|\Supp(X)|$.

For a finite set $I$, we refer to distributions $X \in
\ProbDist(\Omega^I)$ as called {\em $I$-indexed joint distribution
(with values in $\Omega$)}. We view the coordinates $X_i$ ($i \in
I$) as random variables taking values in $\Omega$. We speak of
{\em independence} and {\em conditional independence} among random
variables $X_i$. For all $J \subseteq I$, we denote by $X_J$ the
{\em marginal $J$-indexed joint distribution} $\sq{X_j : j \in J}$
viewed as a distribution in $\ProbDist(\Omega^J)$. 

For an $I$-indexed joint distribution $X$, we denote by $h_X : \wp(I) \lra [0,1]$ the {\em normalized entropy function of $X$} defined by $h_X(J) = \H(X_J)/\H(X)$. By Shannon's classical information inequalities (see \cite{Yeung}), the function $h_X$ is monotone and submodular.

For a graph $G$, a $V_G$-indexed joint distribution $X \in \ProbDist(\Omega^{V_G})$ is a {\em Markov random field over $G$} if $\H(X_A) + \H(X_B) = \H(X_{A \cup B}) + \H(X_{A \cap B})$ for all $A,B \subseteq V_G$ such that $A \cap B$ separates $A \setminus B$ and $B \setminus A$ in $G$.
By Shannon's information inequalities, for $X \in \MRF(G)$, the function $A \longmapsto \H(X_A)$ is $G$-polymatroidal (recall Definition~\ref{df:G-polymat}). Hence, assuming $\H(X) > 0$, the normalized entropy function $h_X$ belongs to $\mc P(G)$. By Lemma~\ref{la:entropy-of-chordal-MRF}, it follows that
\begin{equation}\label{eq:entropy-of-chordal-MRF}
    \H(X) = 
    \sum_{S \subseteq \MaxCliques(G)} -(-1)^{|S|} \H(X_{\cap S}).
\end{equation}

We denote by $\MRF(G,\Omega)$ the set of all Markov random fields
over $G$ with values in $\Omega$. We write $\MRF(G)$ for the class
of all Markov random fields over $G$. Note that $\MRF(G)$ depends
only on the underlying simple graph of $G$. If $G_1$ and $G_2$
are simple graphs such that $V_{G_1} = V_{G_2}$ and $E_{G_1}
\supseteq E_{G_2}$, then $\MRF(G_1) \subseteq \MRF(G_2)$, i.e.,
every Markov random field over $G_1$ is a Markov random field over
$G_2$.

\begin{ex}\label{ex:entropic}
For all graphs $G$ and $T$ such that $G \to T$, the uniform distribution on $\Hom(G,T)$, viewed as an element of $\ProbDist((V_T)^{V_G})$, is a Markov random field over $G$ with entropy $\log\hom(G,T)$.
\end{ex}

The next lemma gives a mechanism for constructing one Markov random field from another.

\begin{la}[Pullback of a MRF]\label{la:pullback}
Let $\varphi$ be a homomorphism from a chordal graph $F$ to a graph $G$. 
Then for every $X \in \MRF(G,\Omega)$ there
exists a unique $\wt X \in \MRF(F,\Omega)$ (called the {\em pullback
of $X$ along $\varphi$}) such that for every clique $C \in
\Cliques(F)$, the marginal distributions $\sq{\wt X_c : c \in C}$ and
$\sq{X_{\varphi(c)} : c \in C}$ are identical. Moreover, if $\Omega
= V_T$ where $T$ is a graph such that $\Supp(X) \subseteq \Hom(G,T)$,
then $\Supp(\wt X) \subseteq \Hom(F,T)$.
\end{la}

We already saw pullbacks of Markov random fields in action when we
computed $\hde(\Vee,\Tri)$ in \S\ref{sub:method}.

\begin{proof}[Proof Sketch]
We can construct $\wt X$ according to the following procedure. Fix
an arbitrary elimination ordering $v_1,\dots,v_n$ of $F$ (so that
$v_j$ is an eliminable vertex of $F|_{\{v_1,\dots,v_j\}}$ for all $j
\in [n]$). We now pick values for $\wt X_{v_1},\dots,\wt X_{v_n}$
(i.e., the coordinates of joint distribution $\wt X = (\wt X_v)_{v
\in F} \in \ProbDist(\Omega^{V_F})$) in order. Assuming values $\wt
X_{v_1},\dots,\wt X_{v_{j-1}}$ have been picked, we next pick $\wt
X_{v_j}$ according to the distribution $X_{\varphi(v_j)}$
conditioned on $X_{\varphi(v_i)} = \wt X_{v_i}$ for $i=1,\dots,j-1$.

One can show that the resulting distribution $\wt X$ is a Markov
random field over $F$. Indeed, it is the unique Markov random field
meeting the conditions of the lemma; in particular $\wt X$ is
independent of the particular elimination ordering $v_1,\dots,v_n$
of $F$. In the event that $\Omega = V_T$ where $T$ is a graph such
that $\Supp(X) \subseteq \Hom(G,T)$, it is easy to show that every
point of $(V_T)^{V_F}$ in the support of $\wt X$ is a homomorphism
in $\Hom(F,T)$.
\end{proof}

\section{Proof of Theorem~\ref{thm:lower} (HDE Lower Bound for Chordal $F$)}\label{sec:lb}

Suppose $F$ is chordal and $\Hom(F,G)$ is nonempty. Let $T$ be a graph
such that $\hom(G,T) \ge 2$. Let $X \in \ProbDist((V_T)^{V_G})$ be the uniform distribution
on $\Hom(G,T)$ (so $X \in \MRF(G)$, see Example~\ref{ex:entropic}).  
Let $h_X : \wp(V_G) \lra [0,1]$ be the
normalized entropy function of $X$ and note that $h_X \in \mc P(G)$ and
\begin{equation*}
    h_X(A) = 
    \H(X_A) / \log\hom(G,T).\label{eq:h0}
\end{equation*}

For each homomorphism $\varphi \in \Hom(F,G)$, let $Y^\varphi \in
\MRF(F,V_T)$ be the pullback of $X$ along $\varphi$, as described in
Lemma~\ref{la:pullback}. We have $\Supp(Y^\varphi) \subseteq
\Hom(F,T)$ and hence $\H(Y^\varphi) \le \log \hom(F,T)$.

By equation (\ref{eq:entropy-of-chordal-MRF}) we have the following
identity (independent of the graph $T$):
\begin{equation*}
    \H(Y^\varphi) = \sum_{S \subseteq \MaxCliques(F)} -(-1)^{|S|}
    \H(X_{\varphi(\cap S)}) = \sum_{S \subseteq \MaxCliques(F)} -(-1)^{|S|}
    h_X(\varphi({\ts\bigcap} S))\H(X).
\label{hmanip}
\end{equation*}
It follows that
\begin{equation*}\label{eq:lower}
    \log\hom(F,T) \ge \max_{\varphi \in \Hom(F,G)}
     \sum_{S \subseteq \MaxCliques(F)} -(-1)^{|S|}
    h_X(\varphi({\ts\bigcap} S))\log\hom(G,T).
\end{equation*}
Since this inequality holds for all graphs $T$ such that $\hom(G,T)
\ge 2$, we have
\begin{align*}
    \hde(F,G) &=
    \inf_{T\,:\,\hom(G,T) \ge 2}
    \frac{\log\hom(F,T)}{\log\hom(G,T)}\quad\text{(by (\ref{hdeotherdef}))}\\
    &\ge
    \inf_{T\,:\,\hom(G,T) \ge 2}
    \max_{\varphi \in \Hom(F,G)}  \sum_{S \subseteq \MaxCliques(F)} -(-1)^{|S|}
    h_X(\varphi({\ts\bigcap} S)).
\end{align*}
Since $h_X \in \mc P(G)$ for all $T$, we get the desired result
that
\[
    \hde(F,G) \ge \min_{p \in \mc P(G)}
    \max_{\varphi \in \Hom(F,G)}  \sum_{S \subseteq \MaxCliques(F)} -(-1)^{|S|}
    p(\varphi({\ts\bigcap} S)).\qedhere
\]

\section{Proof of Theorem~\ref{thm:upper} (HDE Upper Bound)}\label{sec:ub}

Fix a graph $G$ and a function $q \in \mc Q(G)$. That is, let $q$ be
a function from $\wp(V_G)$ to $[0,1]$ such that $q(\emptyset) = 0$
and $\sum_{A \subseteq V_G} q(A) \cdot \components(G|_A) = 1$.

We define a sequence $(T_n)_{n \ge 1}$ of ``target'' graphs as
follows. Vertices of $T_n$ are all pairs $(x,i)$ where $x \in V_G$
and $i \in \N^{\{A \subseteq V_G : x \in A\}}$ is a function from
$\{A \subseteq V_G : x \in A\}$ to $\N$ which satisfies $i(A) <
n^{q(A)}$. There is an edge in $T_n$ from vertex $(x,i)$ to vertex
$(y,j)$ if and only if $(x,y) \in E_G$ and $i(A) = j(A)$ for all
$\{x,y\} \subseteq A \subseteq V_G$.

Let $\pi_n$ denote the homomorphism from $T_n$ to $G$ defined by
$\pi_n((x,i)) = x$. Let $F$ be a graph and suppose $\varphi$ is a
homomorphism from $F$ to $G$. We denote by $\Hom_\varphi(F,T_n)$ the
set of homomorphisms $\psi : F \to T_n$ such that $\pi_n \circ \psi
= \varphi$, i.e., the following diagram commutes:
\[
    \xymatrix{&T_n  \ar[d]^{\pi_n} \\ F \ar[ru]^\psi \ar[r]^\varphi &
    G}
\]
Let $\hom_\varphi(F,T_n) = |\Hom_\varphi(F,T_n)|$ and note
that
\begin{equation}\label{eq:hom}
    \hom(F,T_n) = \sum_{\varphi \in \Hom(F,G)}\hom_\varphi(F,T_n).
\end{equation}

\begin{la}\label{la:hom_varphi}
$\ds\lim_{n \to \infty} \log_n \hom_\varphi(F,T_n) = \sum_{A
\subseteq V_G} q(A) \cdot \components(F|_{\varphi^{-1}(A)}).$
\end{la}

\begin{proof}
Let $\psi \in \Hom_\varphi(F,T_n)$. Each vertex $u \in V_F$ is
mapped under $\psi$ to a pair $(\varphi(u),i_u)$ for some
$i_u \in \N^{\{A \subseteq V_G : \varphi(u) \in A\}}$ subject to
$i_u(A) < n^{q(A)}$. The family of functions $(i_u)_{u \in V_F}$ is
further subject to the constraint that $i_u(A) = i_v(A)$ for all
$u,v \in V_F$ and $\{\varphi(u),\varphi(v)\} \subseteq A \subseteq
V_G$ such that $u$ and $v$ lie in the same connected component of
$F|_{\varphi^{-1}(A)}$. To see this, consider an undirected path in
$F|_{\varphi^{-1}(A)}$ from $u$ to $v$, i.e., a sequence $u =
w_0,w_1,w_2,\dots,w_k = v$ such that $(w_{\ell-1},w_\ell)$ or
$(w_\ell,w_{\ell-1})$ is an edge in $F|_{\varphi^{-1}(A)}$ for every
$\ell \in \{1,\dots,k\}$. Suppose $\{\varphi(u),\varphi(v)\}
\subseteq A \subseteq V_G$ and $u,v$ lie in the same connected
component of $F|_{\varphi^{-1}(A)}$. Then clearly
$\{\varphi(w_{\ell-1}),\varphi(w_\ell)\} \subseteq A$ for all $\ell
\in \{1,\dots,k\}$. Since $(w_{\ell-1},w_\ell)$ or
$(w_\ell,w_{\ell-1})$ is an edge in $F$ and $\psi$ is a homomorphism
from $F$ to $T_n$, we have that $(\psi(w_{\ell-1}),\psi(w_\ell))$ or
$(\psi(w_\ell),\psi(w_{\ell-1}))$ is an edge in $T_n$. It follows
that $i_{\varphi(w_{\ell-1})}(B) = i_{\varphi(w_\ell)}(B)$ for all
$\{\varphi(w_{\ell-1}),\varphi(w_\ell)\} \subseteq B \subseteq V_G$.
In particular, we have $i_{\varphi(w_{\ell-1})}(A) =
i_{\varphi(w_\ell)}(A)$. Therefore $i_u(A) = i_{w_0}(A) = \dots =
i_{w_k}(A) = i_v(A)$.

Conversely, every family of functions $\sq{j_u \in \N^{\{A \subseteq
V_G : \varphi(u) \in A\}} : u \in V_F}$
subject to $j_u(A) < n^{q(A)}$ and $j_u(A) = j_v(A)$ for all $u,v
\in V_F$ and $\{\varphi(u),\varphi(v)\} \subseteq A \subseteq V_G$
such that $u$ and $v$ lie in the same connected component of
$F|_{\varphi^{-1}(A)}$, determines a distinct homomorphism in
$\Hom_\varphi(F,T_n)$. Thus, $\hom_\varphi(F,T_n)$ equals the number
of such families $(j_u)_{u \in V_F}$. This is precisely $\prod_{A
\subseteq V_G} \big \lceil n^{q(A) \cdot
\components(F|_{\varphi^{-1}(A)})} \big \rceil$, since for each $A
\subseteq V_G$ and each connected component $U$ of
$F|_{\varphi^{-1}(A)}$, we have an independent choice of numbers
$m_{A,U} \in \{0,\dots,\lceil n^{q(A)} \rceil - 1\}$ such that
$j_u(A) = m_{A,U}$ for all $u \in U$. Taking logarithms in base $n$,
we get the statement of the lemma.
\end{proof}

\begin{cor}\label{cor:hom}
$\ds\lim_{n \to \infty} \log_n \hom(F,T_n) = \max_{\varphi \in
\Hom(F,G)} \sum_{A \subseteq V_G} q(A) \cdot
\components(F|_{\varphi^{-1}(A)}).$
\end{cor}

This corollary follows immediately from (\ref{eq:hom}) and Lemma~\ref{la:hom_varphi}. We are ready to prove Theorem~\ref{thm:upper}.

\begin{proof}[Proof of Theorem~\ref{thm:upper}]
Suppose $F \to G$. For $q \in \mc Q(G)$, let $(T_n)_{n \ge 1}$ be the sequence of ``target'' graphs as above. By
Corollary~\ref{cor:hom} (applied to $G$), we have
\begin{equation*}\label{eq:corG}
    \lim_{n \to \infty} \log_n \hom(G,T_n) = \max_{\varphi \in
    \Hom(G,G)} \sum_{A \subseteq V_G} q(A) \cdot
    \components(G|\varphi^{-1}(A))
    \ge \sum_{A \subseteq V_G} q(A) \cdot
    \components(G|_A) = 1
\end{equation*}
where the middle inequality is obtained by taking $\varphi$ to be the identity homomorphism on $G$.

We now have
\[
    \hde(F,G) \stackrel{(\ref{hdeotherdef})}{\le}
    \lim_{n \to \infty}\frac{\log_n\hom(F,T_n)}{\log_n\hom(G,T_n)}
    \le 
    \lim_{n \to \infty}\log_n\hom(F,T_n)
    = \max_{\varphi \in
    \Hom(F,G)} \sum_{A \subseteq V_G} q(A) \cdot
    \components(F|_{\varphi^{-1}(A)})
\]
where the last equality is by Corollary~\ref{cor:hom}. Since this inequality holds for all $q \in \mc Q(G)$, it follows that
\[
    \hde(F,G) \le \min_{q \in \mc Q(G)} \max_{\varphi \in
    \Hom(F,G)} \sum_{A \subseteq V_G} q(A) \cdot
    \components(F|_{\varphi^{-1}(A)}).\qedhere
\]
\end{proof}

\section{Proof of Theorem~\ref{thm:tight} (HDE of Chordal $F$ and Series-Parallel $G$)}\label{sec:tight}

Suppose $F$ is chordal and $G$ is series-parallel and $F \to G$. The HDE lower bound of Theorem~\ref{thm:lower} states\begin{align*}
  \hde(F,G) &\ge \min_{p \in \mc P(G)} \max_{\varphi \in
  \Hom(F,G)} \sum_{S \subseteq
  \MaxCliques(F)} -(-1)^{|S|} \cdot p(\varphi({\ts\bigcap} S)).
\end{align*}
Let $p$ be an arbitrary function in $\mc P(G)$. To prove Theorem~\ref{thm:tight} (i.e., to prove this inequality is tight), we construct a sequence of graphs $T_n$ satisfying
\begin{align}
\lim_{n \to \infty} \log_n\hom(G,T_n) &\ge 1,\label{eq:G-T}\\
\lim_{n \to \infty} \log_n\hom(F,T_n) &\le \max_{\varphi \in \Hom(F,G)} \sum_{S \subseteq \MaxCliques(F)} -(-1)^{|S|}p(\varphi({\ts\bigcap}S)).\label{eq:F-T}
\end{align}
Tightness of the above HDE lower bound then follows from (\ref{hdeotherdef}).

To simplify matters, we first consider the special case that $G$ is chordal. (Since $G$ is chordal and series-parallel, it has clique number $\le 3$, i.e., $G$ is a $2$-tree.) After proving Theorem~\ref{thm:tight} in this special case, we give the argument for general series-parallel $G$ in Section~\ref{sub:series-parallel}.

We construct $T = T_n$ in two stages. For every $A \in \MaxCliques(G)$, we construct a graph $T_A$ together with a homomorphism $\pi_A : T_A \lra K_A$ (the complete graph on $A$, viewed as a subgraph of $G$). We then patch together (via a randomized gluing procedure) the various graphs $T_A$ into a graph $T$ together with a homomorphism $\pi : T \lra G$. (This indexing over maximal cliques in the chordal graph $G$ is essential to defining the gluing procedure in a consistent fashion.)

For $a,b,c \in V_G$, we write $p(a),p(ab),p(abc)$ for $p(\{a\}),p(\{a,b\}),p(\{a,b,c\})$ respectively. For $A \subseteq V_G$, we treat $n^{p(A)}$ as integers (by rounding), mindful to preserve identities such as $n^{p(a)+p(bc)} = n^{p(a)}n^{p(bc)}$. Because we are ultimately interested in asymptotics in log base $n$, this kind of rounding presents no difficulties.

\subsection{Construction of $T_A$}

Consider any $A \in \MaxCliques(G)$ and note that $|A| \in \{1,2,3\}$. 

If $|A| = 1$ (say $A = \{a\}$), then $T_A$ is the empty (edgeless) graph on $n^{p(a)}$ vertices and $\pi_A$ maps all vertices of $T_A$ to $a$.

Now suppose $|A| = 2$ (say $A = \{a,b\}$). Letting\begin{equation}\label{eq:alpha-beta-gamma}
\alpha = n^{p(a)}, \quad \beta = n^{p(b)}, \quad \gamma = n^{p(a) + p(b) - p(ab)}
\end{equation}
(note that $\gamma \ge 1$ by submodularity of $p$), $T_A$ is the graph $\gamma {\cdot} K_{\alpha,\beta}$ (i.e., $\gamma$ disjoint copies of the complete bipartite graph $K_{\alpha,\beta}$) and $\pi_A \in \Hom(T_A,K_A)$ maps the two parts of each $K_{\alpha,\beta}$ to vertices $a$ and $b$ of $K_A$ (i.e., the $\alpha$-size part to $a$ and the $\beta$-size part to $b$).

We now examine the nontrivial case when $|A| = 3$ (say $A = \{a,b,c\}$). Consider the restriction of $p$ to $\wp(A)$. So long as $p(A) > 0$, the normalized function $\frac{p}{p(A)} \uhr \wp(A)$ is $K_A$-polymatroidal (if $p(A)=0$, then $p \uhr \wp(A)$ is identically zero). By Example~\ref{ex:K_3}, it follows that $p \uhr \wp(A)$ is a nonnegative linear combination of functions $f_a,f_b,f_c,f_{ab},f_{ac},f_{bc},f_{abc}$ and $f_{\mr{RS}}$. That is,
\[
	p \uhr \wp(A) = \sum_{i \in \{a,b,c,ab,ac,bc,abc,\mr{RS}\}} \lambda_i f_i \text{ for some } \lambda_i \ge 0.
\]
(We will harmlessly treat $n^{\lambda_i}$ as integers.) Note the identities:
\begin{align}
  p(a) &= \lambda_a + \lambda_{ab} + \lambda_{ac} + \lambda_{abc} + \ts\frac{1}{2}\lambda_{\mr{RS}},\notag\\  
  p(ab) &= \lambda_a + \lambda_b + \lambda_{ab} + \lambda_{ac} + \lambda_{bc} + \lambda_{abc} + \ts\frac{1}{2}\lambda_{\mr{RS}},\label{eq:lambda}\\
  p(abc) &= \lambda_a + \lambda_b + \lambda_c + \lambda_{ab} + \lambda_{ac} + \lambda_{bc} + \lambda_{abc} + \ts\frac{1}{2}\lambda_{\mr{RS}}.\notag
\end{align}

For each $i \in \{a,b,c,ab,ac,bc,abc,\mr{RS}\}$, we will construct a graph $T_{A,i}$ and a homomorphism $\pi_{A,i} : T_{A,i} \lra K_A$. Once we have defined these, we obtain $T_A$ as the {\em fibered product} of graphs $T_{A,i}$: 
\begin{itemize}
\item
the vertices of $T_A$ are the elements $(v_i) \in \prod_i T_{A,i}$ such that $\pi_{A,i}(v_i) = \pi_{A,j}(v_j)$ for all $i,j \in \{a,b,c,ab,ac,bc,abc,\mr{RS}\}$, and
\item
there is an edge between vertices $(v_i)$ and $(w_i)$ of $T_A$ if and only if there is an edge between $v_i$ and $w_i$ in $T_{A,i}$ for every $i \in \{a,b,c,ab,ac,bc,abc,\mr{RS}\}$.
\end{itemize}
The homomorphism $\pi_A : T_A \lra K_A$ is defined in the obvious way: 
\begin{itemize}
\item
$\pi_A((v_i))$ equals the common value of $\pi_{A,i}(v_i)$.
\end{itemize}

We now define $T_{A,i}$ and $\pi_{A,i}$ for the various $i \in \{a,b,c,ab,ac,bc,abc,\mr{RS}\}$. In all cases, after defining $T_{A,i}$, the homomorphism $\pi_{A,i}$ will be obvious. Also, the definitions of $T_{A,b}$ and $T_{A,c}$ will be obvious after stating the definition of $T_{A,a}$, so we include only the cases $i \in \{a,ab,abc,\mr{RS}\}$.
\begin{itemize}
\item
$T_{A,a}$ has vertex set $(\{a\} \times [n^{\lambda_a}]) \cup \{b,c\}$ and edges $\{b,c\}$ and $\{(a,i),b\}$ and $\{(a,i),c\}$ for all $i \in [n^{\lambda_a}]$.
\item
$T_{A,ab}$ has vertex set $(\{a,b\} \times [n^{\lambda_{ab}}]) \cup \{c\}$ and edges $\{(a,i),(b,i)\}$ and $\{(a,i),c\}$ and $\{(b,i),c\}$ for all $i \in [n^{\lambda_{ab}}]$.
\item
$T_{A,abc}$ has vertex set $\{a,b,c\} \times [n^{\lambda_{abc}}]$ and edges $\{(a,i),(b,i)\}$ and $\{(a,i),(c,i)\}$ and $\{(b,i),(c,i)\}$ for all $i \in [n^{\lambda_{abc}}]$.
\item
If $\lambda_{\mr{RS}} = 0$, then $T_{A,\mr{RS}} = K_A$ and $\pi_A$ is the identity function on $A$.
\end{itemize}
To define the remaining graph $T_{A,\mr{RS}}$ when $\lambda_{\mr{RS}} > 0$, we use a result of Ruzsa and Szemer\'edi~\cite{RS}.

\begin{thm}[Ruzsa-Szemer\'edi~\cite{RS}]\label{thm:RS}
For all $m \in \N$, there exists a tripartite graph $H(m)$ in which:
\begin{enumerate}[\normalfont\hspace{\parindent}(i)]\setlength{\itemsep}{2pt}
\item
each part has size $m$, 
\item
there are $m^{2-o(1)}$ triangles, and
\item
every edge is contained in exactly one triangle.
\end{enumerate}
\end{thm}
(This is not the usual statement of the Ruzsa-Szemer\'edi result. However, it is easily seen to be equivalent to the usual statement that there exists a bipartite graph with parts of size $m$ whose edge set is the disjoint union of $m^{1-o(1)}$ induced matchings of size at least $m^{1-o(1)}$.) 

Using Theorem~\ref{thm:RS}, we define $T_{A,\mr{RS}}$ in the remaining case:
\begin{itemize}
\item
If $\lambda_{\mr{RS}} > 0$, let $T_{A,\mr{RS}}$ be the graph $H(n^{\frac{1}{2}\lambda_{\mr{RS}}})$ of Theorem~\ref{thm:RS} and let $\pi_{A,\mr{RS}} \in \Hom(T_{A,\mr{RS}},K_A)$ be any function mapping the three parts to $a$, $b$ and $c$.
\end{itemize}

Recalling the definition of $T_A$ (as a fibered product of graphs $T_{A,i}$), it is easy to check using equations (\ref{eq:lambda}) that the graph $T_A$ satisfies:
\begin{align*}
  |\{\text{vertices of $T_A$ which map to $a$ under $\pi_A$}\}| &= n^{p(a)},\\ 
  |\{\text{edges of $T_A$ which map to $\{a,b\}$ under $\pi_A$}\}| &= n^{p(ab) - o(1)},\\
  |\{\text{triangles in $T_A$}\}| &= n^{p(abc)-o(1)}. 
\end{align*}
Moreover, the $o(1)$ terms disappear whenever $\lambda_{\mr{RS}} = 0$.

\subsection{Gluing Procedure}

We now describe the randomized procedure for gluing together the various graphs $T_A$ and homomorphisms $\pi_A : T_A \lra K_A$ into a single graph $T$ and homomorphism $\pi : T_A \lra G$. It is enough to describe the procedure for gluing a pair of graphs $T_A$ and $T_B$ for $A,B \in \MaxCliques(G)$: there is an obvious way of simultaneously and consistently carrying out all pairwise gluings to obtain $T$ and $\pi$ (relying on the chordality of $G$).

Let $A,B \in \MaxCliques(G)$. There are three gluing procedures to consider, depending on $|A \cap B| \in \{0,1,2\}$. In the simplest case that $A \cap B = \emptyset$, the gluing of $T_A$ and $T_B$ is just the disjoint union $T_A \uplus T_B$ and gluing of homomorphisms $\pi_A$ and $\pi_B$ is obvious.

Next suppose that $|A \cap B| = 1$ (say $A \cap B = \{a\}$). Note that $|\pi_A^{-1}(a)| = |\pi_B^{-1}(a)| = n^{p(a)}$. The gluing of $T_A$ and $T_B$ is defined by starting with the disjoint union $T_A \uplus T_B$ and identifying pairs of vertices in $\pi_A^{-1}(a) \times \pi_B^{-1}(a)$ under a uniformly choosen random bijection between sets $\pi_A^{-1}(a)$ and $\pi_B^{-1}(a)$. 

Finally, suppose that $|A \cap B| = 2$ (say $A \cap B = \{a,b\}$). In this case, it must happen that $|A| = |B| = 3$. Define $\alpha,\beta,\gamma$ again by equation (\ref{eq:alpha-beta-gamma}) and consider the graph $\gamma{\cdot}K_{\alpha,\beta}$. We claim that bipartite graphs $T_A|_{\pi^{-1}_A(\{a,b\})}$ and $T_B|_{\pi^{-1}_B(\{a,b\})}$ both look like $\gamma{\cdot}K_{\alpha,\beta}$ after deleting an $n^{-o(1)}$-fraction of edges from the latter. 
(The proof of Claim~\ref{claim1}, below, follows easily from definitions.)

\begin{claim}\label{claim1}
There exist homomorphisms $\xi_A : T_A|_{\pi^{-1}_A(\{a,b\})} \lra \gamma{\cdot}K_{\alpha,\beta}$ and $\xi_B : T_A|_{\pi^{-1}_B(\{a,b\})} \lra \gamma{\cdot}K_{\alpha,\beta}$ such that
\begin{itemize}\setlength{\itemsep}{2pt}
\item
$\xi_A$ and $\xi_B$ are bijections (between vertex sets), and
\item
$\xi_A$ maps $\pi^{-1}_A(a)$ to the $\alpha$-side of $\gamma{\cdot}K_{\alpha,\beta}$ and $\pi^{-1}_A(b)$ to the $\beta$-side of $\gamma{\cdot}K_{\alpha,\beta}$, and similarly for $\xi_B$.
\end{itemize}
Moreover, $T_A|_{\pi^{-1}_A(\{a,b\})}$ and $T_B|_{\pi^{-1}_B(\{a,b\})}$ both have at least $n^{\alpha+\beta+\gamma-o(1)}$ edges (thus, these graphs may be obtained from $\gamma{\cdot}K_{\alpha,\beta}$ by deleting an $n^{-o(1)}$-fraction of edges).
\end{claim}

After fixing arbitrary $\xi_A$ and $\xi_B$, the gluing procedure works as follows. We pick a uniform random automorphism $\Psi$ of $\gamma{\cdot}K_{\alpha,\beta}$ (i.e., an element of the group $(S_\alpha \times S_\beta) \ltimes S_\gamma$). The function $\xi_B^{-1} \circ \Psi \circ \xi_A$ is a bijection of sets $\pi^{-1}_A(\{a,b\})$ and $\pi^{-1}_B(\{a,b\})$. Starting from the disjoint union of $T_A$ and $T_B$, we identify pairs of vertices under this bijection. Finally, we keep edges between pairs of identified vertices if and only if edges existed between these vertices in both $T_A$ and $T_B$. (Intuitively, we randomly overlap $T_A$ and $T_B$ within the confines of $\gamma{\cdot}K_{\alpha,\beta}$ and keep only the edges which occur in both $T_A$ and $T_B$.)

Having defined randomized gluings for pairs of graphs $T_A$ and $T_B$, suffice it to say that these pairwise gluings can without difficulty be carried out simultaneously and consistently over all $A \in \MaxCliques(G)$ to obtain the graph $T$ and homomorphism $\pi : T \lra G$ (chordality of $G$ is crucial here).

\subsection{Counting Homomorphisms from $F$ and $G$}

Now that we have defined the sequence of graphs $T_n$ and homomorphisms $\pi_n : T_n \lra G$, it remains to prove inequalities (\ref{eq:G-T}) and (\ref{eq:F-T}). Both inequalities follow from the following claim.
\begin{claim}\label{claim:H}
If $H$ is a chordal graph and $\varphi \in \Hom(H,G)$, then
\[
	\log_n|\{\theta \in \Hom(H,T_n) : \pi_n \circ \theta = \varphi\}| = \sum_{S \subseteq \MaxCliques(H)} -(-1)^{|S|} p(\varphi({\ts\bigcap}S)) - o(1).
\]
\end{claim}

Before proving Claim~\ref{claim:H}, let's see how it implies inequalities (\ref{eq:G-T}) and (\ref{eq:F-T}). To prove (\ref{eq:G-T}), we take $H=G$ and $\varphi = \mr{id}_{V_G}$ (the identity map on $V_G$ viewed as a homomorphism $G \lra G$) in Claim~\ref{claim:H} and see that
\begin{align*}
\log_n\hom(G,T_n) &\ge \log_n|\{\theta \in \Hom(G,T_n) : \pi_n \circ \theta = \mr{id}_{V_G}\}|\\
&= \sum_{S \subseteq \MaxCliques(G)} -(-1)^{|S|} p({\ts\bigcap}S) - o(1) 
= 1 - o(1)\quad\text{(by Lemma~\ref{la:entropy-of-chordal-MRF})}.
\end{align*}
Inequality (\ref{eq:F-T}) is immediate from Claim~\ref{claim:H} taking $H=F$:
\begin{align*}
\lim_{n\to\infty}\log_n\hom(F,T_n) &= \lim_{n\to\infty}\max_{\varphi \in \Hom(F,G)} \log_n|\{\theta \in \Hom(F,T_n) : \pi_n \circ \theta = \varphi\}| \quad\text{(as $\hom(F,T_n) \xrightarrow{n \to \infty} \infty$)}\\
&= \sum_{S \subseteq \MaxCliques(H)} -(-1)^{|S|} p(\varphi({\ts\bigcap}S)).
\end{align*}

Now for the proof of this claim:
\begin{proof}[Proof of Claim~\ref{claim:H}]
We define a supergraph $T^\ast$ of $T$ as follows. For each $A \in \MaxCliques(G)$, we define a supergraph $T_A^\ast$ of $T_A$ and apply the same gluing procedure. If $|A| \le 2$, let $T_A^\ast = T_A$. If $|A| = 3$ (say $A = \{a,b,c\}$), recall that $T_A$ is the fibred product of graphs $T_{A,a},\dots,T_{A,abc}$ and $T_{A,\mr{RS}}$; let $T_A^\ast$ be the fibred product of graphs $T_{A,a},\dots,T_{A,abc}$ and $T_{A,\mr{RS}}^\ast$ where $T_{A,\mr{RS}}^\ast$ is the complete tripartite graph with all parts of size $n^{\frac{1}{2}\lambda_{\mr{RS}}(A)}$. Viewing $T_{A,\mr{RS}}$ as a subgraph of $T_{A,\mr{RS}}^\ast$ (with the same vertex set) and apply the same gluing procedure (i.e., with the same randomization), we view $T$ as a subgraph of $T^\ast$ (with the same vertex set). It now suffices to prove the following:
\begin{align}
&\log_n|\{\theta \in \Hom(H,T^\ast_n) : \pi_n \circ \theta = \varphi\}| =\label{eq:blep1}\\
&\hspace{1.32in}\sum_{S \subseteq \MaxCliques(H)} -(-1)^{|S|} p(\varphi({\ts\bigcap}S))\notag\\
&\hspace{1in}+\sum_{\substack{A \in \MaxCliques(G)\,:\,|A|=3\\ \phantom{\ }}} {\textstyle\frac{1}{2}}\lambda_{\mr{RS}}(A) \cdot |\{A' \in \MaxCliques(H) : \varphi(A') = A\}|,\notag\\
&\log_n\textstyle\Pr_{\theta \in \Hom(H,T^\ast_n)}[\theta \in \Hom(H,T_n)] =\label{eq:blep2}\\
&\hspace{1in}-\sum_{A \in \MaxCliques(G)\,:\,|A|=3} {\textstyle\frac{1}{2}}\lambda_{\mr{RS}}(A) \cdot |\{A' \in \MaxCliques(H) : \varphi(A') = A\}| - o(1).\notag
\end{align}
We first give the argument for equation (\ref{eq:blep1}). Note the following:
\begin{itemize}\setlength{\itemsep}{0pt}
\item
for every edge $(a,b)$ in $G$ and every $a' \in \pi_n^{-1}(a)$,
\[
|\{b' \in \pi_n^{-1}(b) : (a',b')\text{ is an edge in }T_n^\ast\}| = n^{p(ab)-p(a)},
\]
\item
for every triangle $(a,b,c)$ in $G$ and every $a' \in \pi_n^{-1}(a)$ and $b' \in \pi_n^{-1}(b)$ such that $(a',b')$ is an edge in $T_n^\ast$,
\[
|\{c' \in \pi_n^{-1}(c) : (a',b',c')\text{ is a triangle in }T_n^\ast\}| = n^{p(abc) - p(ab) + \frac{1}{2} \lambda_{\mr{RS}}(abc)}.
\]
\end{itemize}
It follows that if $v_1,\dots,v_n$ is an elimination ordering for $H$ then
\begin{align*}
&\log_n|\{\theta \in \Hom(H,T^\ast_n) : \pi_n \circ \theta = \varphi\}| =\\
&\hspace{.15in}\sum_{i=1}^n p\big(\varphi(\{\textup{neighbors of $v_i$ among }v_1,\dots,v_{i-1}\}
    \cup \{v_i\})\big) - p\big(
    \varphi(\{\textup{neighbors of $v_i$ among }v_1,\dots,v_{i-1}\})\big)\\
&\hspace{.15in}+\sum_{\substack{A \in \MaxCliques(G)\,:\,|A|=3\\ \phantom{\ }}} {\textstyle\frac{1}{2}}\lambda_{\mr{RS}}(A) \cdot |\{A' \in \MaxCliques(H) : \varphi(A') = A\}|.
\end{align*}
Equation (\ref{eq:blep1}) now follows using Lemma~\ref{la:entropy-of-chordal-MRF}.

For equation (\ref{eq:blep2}), notice that a triangle $(a',b',c')$ over $(a,b,c)$ in $T_n^\ast$ is a triangle in $T_n$ with probability $n^{-\lambda_{\mr{RS}}(abc)-o(1)}$. Now consider a uniform random homomorphism $\theta \in \Hom(H,T_n^\ast)$. For an edge $(x,y)$ in $H$, consider the vertices $z_1,\dots,z_m$ such that $(x,y,z_j)$ are triangles in $H$. The key observation (using chordality of $H$) is that events $\{(\theta(x),\theta(y),\theta(z_j))$ is a triangle in $T_n\}_{j=1,\dots,m}$ are independent conditioned on $\theta(x)$ and $\theta(y)$. By expanding the probability that $\theta \in \Hom(H,T_n)$ conditionally along an elimination ordering, we see that $\theta \notin \Hom(H,T_n)$ with probability $\prod_{\text{triangles $(x,y,z)$ in }H}n^{-\lambda_{\mr{RS}}(\theta(x)\theta(y)\theta(z))-o(1)}$, which proves (\ref{eq:blep2}) and completes the proof of Claim~\ref{claim:H}.
\end{proof}

\subsection{Series-Parallel $G$}\label{sub:series-parallel}

Finally, we prove the theorem for the case when $G$ is series-parallel (but not necessarily chordal). Recall that for every series-parallel graph $G$, there exists a $2$-tree $\wt G$ (i.e., a $K_4$-free chordal graph)  such that $V_G = V_{\wt G}$ and $E_G \subseteq E_{\wt G}$. Fix any such $\wt G$.

Consider any $p \in \mc P(G)$. Note that $\mc P(G) \subseteq \mc P(\wt G)$ (i.e., any normalized $G$-polymatroidal function is also normalized $\wt G$-polymatroidal). Therefore, we can construct graphs $\wt T_n$ with homomorphisms $\pi_n : \wt T_n \lra \wt G$ such that (by Claim~\ref{claim:H} applied to $\wt G$ and $\wt T_n$) for every chordal graph $H$ and $\varphi \in \Hom(H,\wt G)$,
\begin{equation}\label{eq:wt-T}
	\log_n|\{\theta \in \Hom(H,\wt T_n) : \pi_n \circ \theta = \varphi\}| = \sum_{S \subseteq \MaxCliques(H)} -(-1)^{|S|} p(\varphi({\ts\bigcap}S)) - o(1).
\end{equation}

Let $T_n$ be the subgraph of $\wt T_n$ which has the same vertices, but where we keep an edge $(v,w)$ from $\wt T_n$ if and only if $(\pi_n(v),\pi_n(w))$ is an edge of $G$. Note that $\pi_n$ is a homomorphism in $\Hom(T_n,G)$. By (\ref{eq:wt-T}), Claim~\ref{claim:H} now holds (exactly as stated) for $G$ and $T_n$. The proof of inequalities  (\ref{eq:G-T}) and (\ref{eq:F-T}) then follows by the exact same argument.

\section{Proof of Theorem~\ref{thm:P_4} (HDE of $P_4$ and $P_{4n+2}$)}\label{sec:P_4}

In this section we give the proof of Theorem~\ref{thm:P_4} (the equation $\hde(P_4,P_{4n+2}) = (4n+1)/(4n^2+3n+1)$), which was discovered by solving the linear program of Theorem~\ref{thm:tight} for small values of $n$. We include this proof as an illustration of a somewhat exotic phenomenon arising in the study of a simple HDE problem.

Let $P_{4n+2} = (V,E)$ where $V = \{0,1,\dots,4n+1\}$ and $E =
\big\{\{0,1\},\{1,2\},\dots,\{4n,4n+1\}\big\}$. Define function $f : V \lra
\N$ as follows:
\begin{itemize}
  \item
    $f(0) = f(4n+1) = 2n+1$,
  \item
    $f(4k+1)=f(4k+3)=2k+1$ for $k \in \{0,\dots,n-1\}$,
  \item
    $f(4k+2)=f(4k+4)=2n-2k-1$ for $k \in \{0,\dots,n-1\}$.
\end{itemize}
For every $N \in \N$, we define a random graph $T_N = (V_N,E_N)$ as
follows. Let
\[
    V_N = \big\{(v,i) : v \in V,\ i \in \{1,\dots,\lceil N^{f(v)}\rceil\}\big\}.
\]
Independently for all $(v,i),(w,j) \in V_N$, place an edge with probability
\begin{align*}
        \Pr\big[\{(v,i),&(w,j)\} \in E_N\big]\\
        &=
        \begin{cases}
          \frac{1}{N} &\text{if }\{v,w\} = \{4k,4k+1\}\text{ where }k \in \{0,\dots,n-1\},\\
          1
          &\text{if }\{v,w\} = \{4k+r,4k+r+1\} \text{
          where $k \in \{0,\dots,n-1\}$ and $r \in \{1,2,3\}$},\\
          0  &\text{otherwise.}
        \end{cases}
\end{align*}
\begin{figure}[h!]
    \centering
    \includegraphics{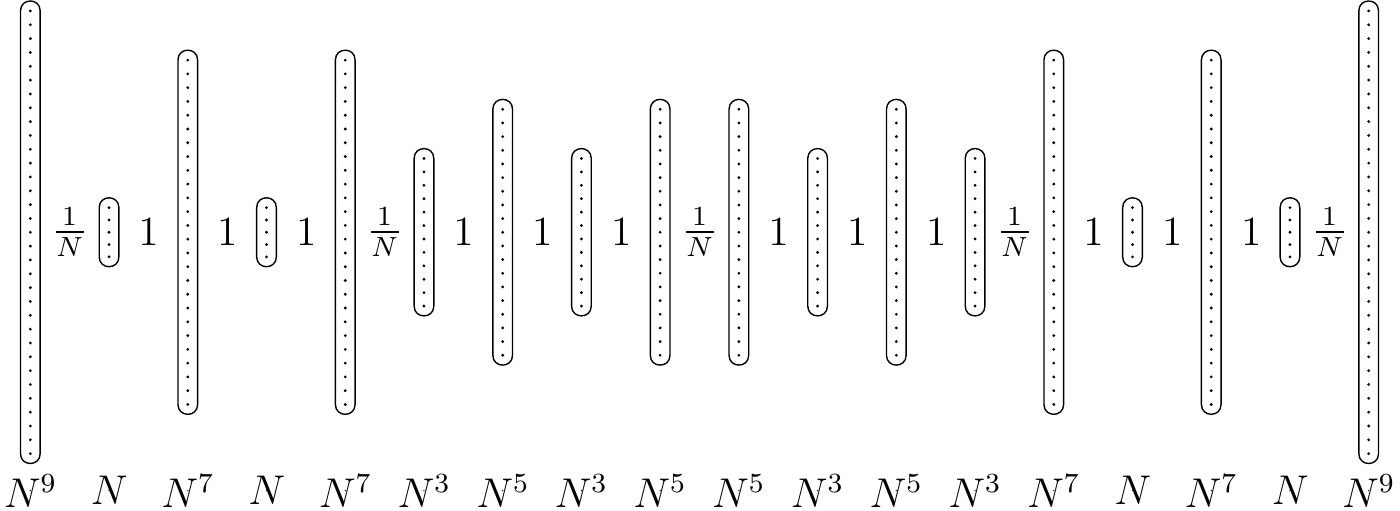}
    \caption{The random graph $T_N$ when $n=4$ (drawn to logscale height).
    The value ($1$ or $\frac 1 N$) in-between partitions of the vertex set indicates the
    probability of an edge.}
\end{figure}

It holds with high probability that
\begin{align*}
    \hom(P_{4n+2},T_N)
    &\ge N^{4n^2+3n+1-o(1)}.
\end{align*}
It also holds with high probability (by inspection of the various homomorphisms
from $P_4$ to $P_{4n+2}$) that
\begin{align*}
    \hom(P_4,T_N) &\le N^{4n+1+o(1)}. 
\end{align*}Therefore,
\[
    \hde(P_4,P_{4n+2}) \le \frac{4n+1}{4n^2 + 3n + 1}.
\]

We now prove the opposite inequality. We will represent
homomorphisms $P_4 \lra P_{4n+2}$ by 4-tuples $\sq{i_1,i_2,i_3,i_4}
\in V^4$. Define a function $w : \Hom(P_4,P_{4n+2}) \lra \N$ as follows:
\begin{align*}
w(\sq{4k,4k+1,4k,4k+1}) &= 1
&&\text{for }k \in \{0,\dots,n\},\\
w(\sq{4k,4k+1,4k+2,4k+1}) &= 1
&&\text{for }k \in \{0,\dots,n-1\},\\
w(\sq{4(n-k)+1,4(n-k),4(n-k)-1,4(n-k)}) &= 1
&&\text{for }k \in \{0,\dots,n-1\},\\
w(\sq{4k+2,4k+3,4k+4,4k+5}) &= 4k+2
&&\text{for }k \in \{0,\dots,n-1\},\\
w(\sq{4(n-k)+1,4(n-k),4(n-k)-1,4(n-k)-2}) &= 4k+2
&&\text{for }k \in \{0,\dots,n-1\},
\end{align*}
and let $w(\varphi) = 0$ for all other homomorphisms $\varphi \in \Hom(P_4,P_{4n+2})$.
Note that $$\sum_{\varphi \in \Hom(P_4,P_{4n+2})} w(\varphi) =
4n^2+3n+1.$$

Fix any target graph $T$ with at least one undirected edge. Let $X
\in \ProbDist((V_T)^{V_G})$ be the uniform distribution on
$\Hom(G,T)$. Let $\Phi$ be a random homomorphism in $\Hom(F,G)$
drawn according to
\[
    \Pr\big[\Phi = \varphi\big] = \frac{w(\varphi)}{4n^2+3n+1}.
\]
Let $Y^\Phi \in \ProbDist((V_T)^{V_F})$ denote the pullback of $X$
along $\Phi$ (so in particular $\Supp(Y^\Phi) \subseteq \Hom(F,T)$).

\begin{figure}[h!]
    \centering
    \includegraphics{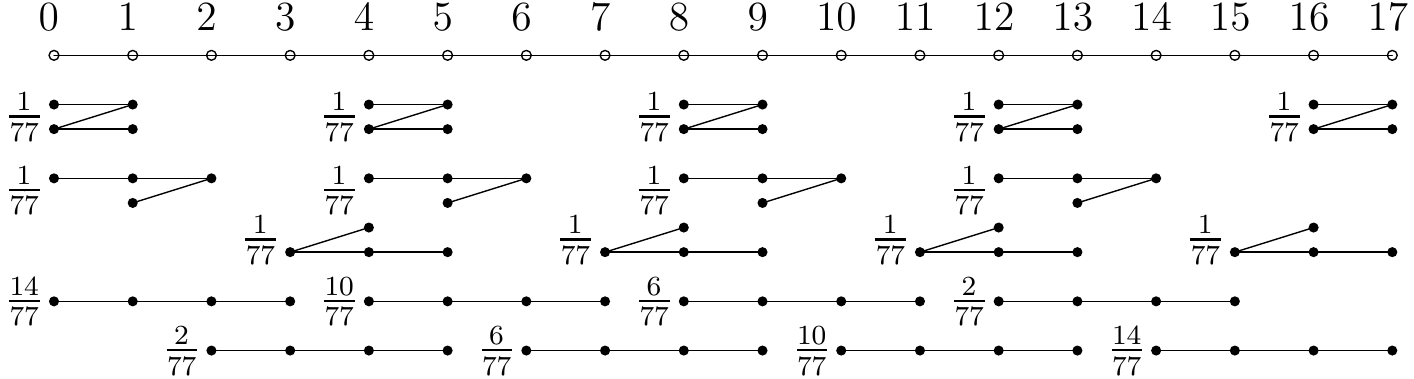}
    \caption{The distribution $\Phi$ of homomorphisms $P_4 \lra P_{4n+2}$ when $n=4$.}
\end{figure}

By a straightforward calculation using equation (\ref{eq:entropy-of-chordal-MRF}), we have
\begin{align}\label{qqq1}
  (4n^2+3n&+1)\H Y^\Phi\\
  &=
   \Big(\H X_{\{0,1\}} - \H X_0\Big) + \Big(\H X_{\{n,n+1\}} - \H X_{4n+1}\Big)
   +\sum_{k=0}^n (4n+1)\H X_{\{4k,4k+1\}}\notag\\
   &\quad+\sum_{k=0}^{n-1}\left(
        \begin{aligned}
            &&(4n-4k)&\H X_{\{4k+1,4k+2\}}\\
            &+&4n &\H X_{\{4k+2,4k+3\}}\\
            &+&(4k+4)&\H X_{\{4k+3,4k+4\}}
        \end{aligned}
        \right)-\left(
        \begin{aligned}
            &&(4n-4k)&\H X_{4k+1}\\
            &+& (4n-4k-1)&\H X_{4k+2}\\
            &+& (4k+3)&\H X_{4k+3}\\
            &+& (4k+4)&\H X_{4k+4}
        \end{aligned}\right).\notag
\end{align}
By monotonicity and submodularity of the entropy operator (also
using the fact that $\H X_\emptyset = 0$), we have
\begin{align}
\label{qqq2}
    0 &\ge
    \left\{
    \begin{aligned}[]
    &\H X_0 - \H X_{\{0,1\}},\\
    &\H X_{4n+1} - \H X_{\{4n,4n+1\}},\\
    &\ts\sum_{k=0}^{n-1} (4k+1)\Big(\H X_{\{4k+1,4k+2\}} - \H X_{4k+1} - \H X_{4k+2}\Big),\\
    &\ts\sum_{k=0}^{n-1} \H X_{\{4k+2,4k+3\}} - \H X_{4k+2} - \H X_{4k+3},\\
    &\ts\sum_{k=0}^{n-1} (4n-4k-3)\Big(\H X_{\{4k+3,4k+4\}} - \H X_{4k+3} -
    \H X_{4k+4}\Big).
    \end{aligned}\right.
\end{align}
Adding each negative quantity in the lefthand side of equation
(\ref{qqq2}) to the righthand side of equation (\ref{qqq1}), we get
\begin{align*}
    (4n^2+3n+1)\H Y^\Phi &\ge
    (4n+1)\left(\sum_{\{v,w\} \in E} \H X_{\{v,w\}} - \sum_{v \in \{1,\dots,4n\}} \H X_v\right)\\
    &= (4n+1)\H X \quad\text{by (\ref{eq:entropy-of-chordal-MRF}).}
\end{align*}
It follows that $\ds\hde(P_4,P_{4n+2}) \ge \frac{4n^2+3n+1}{4n+1}$, as required.

\section{Conclusion}\label{sec:conclusion}

The main open question is whether $\hde(F,G)$ is computable. (This question is equivalent to decidability of the homomorphism domination problem
by virtue of Lemma~\ref{la:homdom-facts}(d).) Theorem~\ref{thm:tight} shows that $\hde(F,G)$ is computable in the special case that $F$ is chordal and $G$ is series-parallel. Examples like $\hde(\Vee,\vec C_3)$
show that the homomorphism domination exponent can be tricky to compute even for very small instances.
Our work also raises the problem of finding a closed-form expression for $\hde(P_m,P_n)$. So far, we only have closed expressions when $m$ is odd or equal to $2$ or $4$. Besides the applications in database theory, we hope that the homomorphism domination exponent will be seen as interesting parameter in its own right.

\subsection*{Acknowledgements}
We thank Madhu Sudan, Noga Alon and Ehud Friedgut for insightful discussions. We also thank an anonymous referee for helpful comments.

\bibliography{homdom}{}
\bibliographystyle{alpha}

\end{document}